\documentclass[12pt, a4paper]{article}
\usepackage{mathtools}
\usepackage{amsthm}
\usepackage{amsmath}
\usepackage{amssymb}
\usepackage{mathrsfs}
\usepackage{geometry}

\usepackage[shortlabels]{enumitem} %To enumerate (list) vertically by (i), (ii), etc.
\usepackage[mathscr]{euscript}
\usepackage{multicol}

\usepackage{parskip} % för att ändra styckesindelning från indrag till rad

\usepackage{tikz-cd} % for commutative diagrams

\usepackage{todonotes}

\usepackage{centernot}

\usepackage{color} % to write with color

\usepackage{authblk} % for authors, affiliation, etc.

\usepackage{appendix} % FÖR APPENDIX(?) (+få med i table of contents (toc))

\usepackage[style=trad-alpha]{biblatex} %Imports biblatex package %trad-plain ger nummer, trad-alpha ger bokstäver
\newtheorem{theorem}{Theorem}[section] %To write a theorem
\theoremstyle{definition}
\newtheorem{definition}[theorem]{Definition} %Definition

\newtheorem{lemma}[theorem]{Lemma}

\numberwithin{equation}{subsection}

\theoremstyle{remark}
\newtheorem{remark}[theorem]{Remark} %Remark
\numberwithin{equation}{section}

\DeclareMathOperator{\im}{im}

\DeclareMathOperator{\Hom}{Hom}

\DeclareMathOperator{\ind}{Ind}

\DeclareMathOperator{\codim}{codim}

\DeclareMathOperator{\dr}{dR} % for de Rham cohomology
\DeclareMathOperator{\ord}{ord} % for order of vanishing

\DeclareMathOperator{\res}{Res} % Weil restriction
\DeclareMathOperator{\fil}{Fil}
\DeclareMathOperator{\Conj}{conj}
\DeclareMathOperator{\sbt}{Sbt} % For the Schubert stack
\DeclareMathOperator{\fr}{Fr} % For the Frobenius
\newcommand{\fkm}{\mathfrak{m}}

\newcommand{\gr}{\text{Gr}}

\newcommand{\gzip}{\text{$G$-$\mathtt{Zip}$}} % for G-Zip
\newcommand{\gzipf}{\text{$G$-$\mathtt{ZipFlag}$}} % for G-ZipFlag

\newcommand{\zip}{\text{-$\mathtt{Zip}$}}
\newcommand{\zipf}{\text{-$\mathtt{ZipFlag}$}}

\newcommand{\calo}{\mathcal{O}}

\newcommand{\call}{\mathcal{L}}

\newcommand{\mz}{\mathbb{Z}}
\newcommand{\mq}{\mathbb{Q}}
\newcommand{\mr}{\mathbb{R}}
\newcommand{\mc}{\mathbb{C}}

\newcommand{\mg}{\mathbb{G}}

\usepackage{hyperref}
\hypersetup{
    colorlinks=true,
    linkcolor=blue,
    filecolor=magenta,      
    urlcolor=cyan,
}
 
\providecommand{\keywords}[1]
{
  \small	
  \textbf{\textit{Keywords---}} #1
}

\title{On the Hasse invariant %and the relative position of the Hodge and conjugate filtrations 
of Hilbert modular varieties mod~$p$.} %tilde ger mellanrum som inte får delas upp
\author{Stefan Reppen\thanks{stefan.reppen@math.su.se}}
\affil{Department of Mathematics, Stockholm University, Stockholm, Sweden}
\date{\vspace{-5ex}}

\begin{document}

\maketitle

\begin{abstract}
    %Let $F$ be a totally real field of degree $n$, let $S$ denote the geometric special fiber of a Hilbert modular variety associated to $F$, at a prime unramified in $F$. Let further $\pi \colon A\to S$ denote the universal abelian variety over $S$ and let $\fil^{\bullet}H_{\dr}^{n}$ and ${}_{\Conj}\fil_{\bullet}H_{\dr}^{n}$ denote the Hodge respectively the conjugate filtration on $H_{\dr}^{n}(A/S)$. We show that the order of vanishing of the Hasse invariant on $S$ at a point $s\in S$ is equal to the largest integer $m$ such that ${}_{\Conj}\fil_{0}H_{\dr}^{n}(A_s) \hookrightarrow \fil^{m}H_{\dr}^{n}(A_s)$. 
    Let $F$ be a totally real field and let $S$ denote the geometric special fiber of a Hilbert modular variety associated to $F$, at a prime unramified in $F$. 
    We show that the order of vanishing of the Hasse invariant on $S$ is equal to the largest integer $m$ such that the smallest piece of the conjugate filtration lies in the $m^{\text{th}}$ piece of the Hodge filtration. 
    This result is a direct analogue of Ogus' on families of Calabi--Yau varieties in positive characteristic (see \cite{ogus}). We also show that the order of vanishing at a point is the same as the codimension of the Ekedahl-Oort stratum containing it.
\end{abstract}

%\tableofcontents

%\newpage
\section{Introduction}
%%%%% INTRODUCTION (extra)
The Hasse invariant was introduced by Hasse in the 1930's as a binary invariant of elliptic curves in characteristic $p>0$. Nowadays, the classical Hasse invariant is often described as a modular form mod $p$ of weight $p-1$, i.e., as a section of the $(p-1)^{\text{th}}$ power of the Hodge line bundle on the geometric special fiber of a Siegel modular variety at a prime $p$. Recalling that a Siegel modular variety is a moduli space for certain abelian varieties, the non-vanishing set of the Hasse invariant consists exactly of the ordinary abelian varieties (i.e., those whose $p$-rank is maximal). %Moreover, at a point in the vanishing locus, the dimension of the corresponding abelian variety minus the order of vanishing of the Hasse invariant gives an upper bound for the $p$-rank of the abelian variety.
The Hasse invariant satisfies other desirable properties; its construction is compatible with varying the prime to $p$ level, some power of it lifts to characteristic 0, and some power extends to the minimal compactification. Several authors have also constructed generalisations of the classical Hasse invariant enjoying similar properties to those described above (see e.g., \cite{goldring.nicole}, %\cite{hernandez}, 
\cite{koskivirta.wedhorn}, \cite{goren}). Thanks to these properties, the Hasse invariant and its generalisations have been used succesfully to produce congruences between automorphic forms and automorphic Galois representations in the Langlands program (see e.g., \cite{emerton.et.al}, \cite{taylor}, \cite{wushi.invent}). In fact, the idea of using the Hasse invariant to this end goes back to the work of Deligne--Serre in the 1970's on modular forms of weight 1 (\cite{deligne.serre}). The theory of Hasse invariants has also been used to study the geometry of Shimura varieties mod $p$. For instance, in \cite{wushi.invent} it is shown that the Ekedahl--Oort (E--O) stratification of a Hodge-type Shimura variety is uniformly principally pure (see \cite{goldring.koskivirta.flag} for a definition), and that each stratum in the E--O stratification of the minimal compactification is affine.
 
The definition of the Hasse invariant also makes sense for other families of varieties, such as Calabi--Yau varieties mod $p$. For such a family, Ogus shows in \cite{ogus} that the order of vanishing of the Hasse invariant is closely related to the relative position of the Hodge and conjugate filtrations. More precisely, he shows that the Hasse invariant vanishes to order $m$ at a given point of the base if and only if $m$ is the largest integer such that the smallest piece of the conjugate filtration lies in the $m^{\text{th}}$ piece of the Hodge filtration.

The current paper is inspired by Ogus' result, based on which we ask if a similar statement holds for certain Shimura varieties. We prove a slightly stronger result in the case of Hilbert modular varieties. This gives some evidence that similar statements might hold for all Hodge type Shimura varieties and in upcoming work we show analogous results for some unitary and Siegel Shimura varieties. The method used for the unitary Shimura varieties is the same as the one used in this paper.
%%%%%
%
%In \cite{ogus} Ogus proves that under certain conditions the order of vanishing of the Hasse invariant of a Calabi-Yau family is closely related to the relative position of the Hodge- and conjugate filtrations. More precisely, he shows that the Hasse invariant vanishes to order $m$ at a given point of the base if and only if $m$ is the largest integer such that the smallest piece of the conjugate filtration lies in the $m^{\text{th}}$ piece of the Hodge filtration. Based on this we asked if a similar statement could hold for certain Shimura varieties, and in this note we prove a slightly stronger statement in the case of Hilbert modular varieties. 

In short,
let $F$ denote a totally real field of degree $n = [F : \mq]$.  Let $S$ denote the geometric special fiber of a Hilbert modular variety associated to $F$, at a prime unramified in $F$, and let $\pi \colon A\to S$ denote the universal abelian variety over $S$. 
%Recall that the Hasse invariant is a modular form on $S$ of weight $p-1$, i.e. a section of the $(p-1)^{\text{th}}$ power of the Hodge line bundle. 
%a section of $(\pi_{*}\Omega_{A/S}^{n})^{p-1}$. 
Let also  $\fil^{\bullet}H_{\dr}^{n}$ denote the Hodge filtration on $H_{\dr}^{n}(A/S)$ and let ${}_{\Conj}\fil_{\bullet}H_{\dr}^{n}$ denote the conjugate filtration. In this text we prove the following statement.
\begin{theorem}\label{main.theorem}
For any point $s\in S$, the following are equivalent:
\begin{enumerate}
    \item The point $s$ lies in a codimension $m$ stratum of the Ekedahl--Oort decomposition of $S$.
    \item The Hasse invariant vanishes to order $m$ at $s$.
    \item The integer $m$ is the largest integer such that ${}_{\Conj}\fil_{0}H_{\dr}^{n}(A_s) \subseteq \fil^{m}H_{\dr}^{n}(A_s)$.
\end{enumerate}
\end{theorem}
\begin{remark}
The equivalence of 2. and 3. is the analogue of Ogus' result for families of Calabi--Yau varieties.
\end{remark}
%\begin{remark}
%It should be noted that the equivalence of 1. and 2. follows from the work of Goren \cite{goren} \todo{double check this remark}. We underscore however that our approach is completely different from ibid (see below), and it seems that our method is more suitable for other Shimura varieties (e.g. it allows us to deal with an infinite list of unitary Shimura varieties (see upcoming work)).
%\end{remark}
\begin{remark}
In \cite{geer.katsura} van der Geer and Katsura asked for a geometric interpretation of the largest integer $m$ such that ${}_{\Conj}\fil_{0}H_{\dr}^{n}(X) \subseteq \fil^{m}H_{\dr}^{n}(X)$, for a smooth complete algebraic variety $X$ of dimension $n$. Theorem \ref{main.theorem} together with the aforementioned result of Ogus and our upcoming work suggests that for families of smooth complete algebraic varieties where the Hasse invariant is defined, the order of vanishing of the Hasse invariant could be one geometric interpretation of the integer in question.
\end{remark}
Our approach is quite different from that of Ogus, which relies on crystalline methods and the Gauss--Manin connection. %Firstly, the method is different as Ogus uses crystalline methods and the Gauss-Manin connection whereas 
We instead use the theory of $G$-zips, introduced by Moonen, Pink, Wedhorn and Ziegler (see \cite{moonen.wedhorn}, \cite{pink.wedhorn.ziegler.1} and \cite{pink.wedhorn.ziegler.2}) and further developed by Goldring and Koskivirta (see \cite{goldring.koskivirta.flag}).
In doing so, we can reduce the study to a question about a section of a line bundle on a product of projective lines. %, and the statement becomes almost obvious. 
In particular, we can immediatelly read off the order of vanishing from the corresponding bihomogeneous polynomial, and there is no need to study in detail any local rings of the Shimura variety itself.

%Secondly, Ogus imposes three assumptions on the Calabi Yau varieties that he studies (see \cite{ogus}, p. 36). Let $f : X\to S$ denote a family of Calabi Yau varieties. The first assumption is that the Hodge spectral sequence of the fibers of the family should degenarate at $E_1$ and that $R^{i}f_{*}\Omega_{X/S}^{j}$ is locally free for all $i$ and $j$. The second assumption asks for the Kodaira-Spencer map to be surjective, and the third assumption is that the map
%\begin{equation}
%    \text{Sym}^{j}R^{1}f_{*}T_{X/S} \otimes f_{*}\Omega_{X/S}^{\dim X/S} \to R^{j}f_{*}\Omega_{X/S}^{\dim(X/S)-j}
%\end{equation}
%induced by cup product and interior multiplication is surjective, for all $j$ less than some fixed integer $l$. Although the first condition indeed holds also for Hilbert modular varieties, it is not clear to us to whether the second and third conditions are satisfied in our case. What we have instead is a smooth morphism from our base scheme to the stack of $G$-zips and we take a leap of faith in hoping that this will suffice to transfer Ogus' result to the case of Hilbert modular varieties. It is quite interesting that it does, and it raises the following question; if we have a family of schemes $X/S$ with a map from the base to a stack of $G$-zips, how is the smoothness of this map related to Ogus' conditions explained above?

The outline of the text is as follows. Since the Hasse invariant is pulled back from sections of line bundles on the stack of $G$-zips, the stack of $G$-zip flags and the Schubert stack, we begin by introducing these stacks (Section \ref{section.three.stacks}) and the associated sheaves construction on them (Section \ref{section.ass.sheaves}). Next we briefly recall the notions of Hodge type Shimura varieties and introduce the Hodge character (Section \ref{section.shimura.var}). In the final section we prove the main theorem (Section \ref{section.proof.of.theorem}).
%
%
%
%
%
%
%
%
% NOTATION
%\section{Notation}
%
%
%
%
%
%
% The stack of G-Zips
\section{Three stacks: $\gzip^{\mu}$, $\gzipf^{\mu}$, $\sbt$}\label{section.three.stacks}
In this section, let $p$ be a prime and $k$ an algebraic closure of $\mathbb{F}_p$. Let $G$ be a reductive group over $\mathbb{F}_p$. For any $k$-scheme $X$, let $X^{(p)}$ denote the fiber product of $X$ with $k$ along the absolute Frobenius of $k$.

Given a cocharacter $\mu \colon \mg_{m,k}\to G_k$ we associate a pair of opposite parabolic subgroups, $(P, P^{+})$, as follows; $P$ (respectively $P^{+}$) is the parabolic whose Lie algebra is the sum of the non-positive (respectively non-negative) weight spaces of $\text{Lie}(G)$ under the action of $\text{Ad}\circ \mu$. % $P^{+}\subset G$ is the group whose $k$-points are given by 
%\begin{equation}
%    P^{+}(k) = \{ g \in G(k) : \lim_{t\to 0}\mu(t)g\mu(t)^{-1} \text{ exists } \}.
%\end{equation}
 This defines a Levi subgroup, $L \coloneqq P\cap P^{+}$. We may and will also assume that there is a Borel $B$ and a maximal torus $T$ such that $T_k \subseteq B_k \subseteq P$.

The datum of a reductive group $G$ over $\mathbb{F}_p$ and a cocharacter $\mu \colon \mg_{m,k}\to G_k$ is called a {\bf cocharacter datum}. To such a pair we will associate certain stacks in a manner described below.

\subsection{The stack of $G$-zips and $G$-zip flags}
Let $(G,\mu)$ be a cocharacter datum and let $P,P^{+}$ and $L$ be as above. For an algebraic group $H$ let $R_u(H)$ denote the unipotent radical of $H$. 
Given a $k$-scheme $S$ we have the following definition.
\begin{definition}
A {\bf $G$-zip of type $\mu$} over $S$ is a tuple $\underline{I} = (I, I_{P}, I_{(P^{+})^{(p)}}, \iota)$, where $I$ is a $G_k$-torsor over $S$, $I_{P}, I_{(P^{+})^{(p)}}\subseteq I$ are $P$ respectively $(P^{+})^{(p)}$-torsors, and $\iota : (I_P)^{(p)}/R_u(P)^{(p)} \to I_{(P^{+})^{(p)}}/R_u((P^{+})^{(p)})$ is an isomorphism of $L^{(p)}$-torsors.
\end{definition}
We denote by $\gzip^{\mu}(S)$ the category of $G$-zips of type $\mu$ over $S$. This gives a category fibered over the category of $k$-schemes, which we denote by $\gzip^{\mu}$. It is a smooth stack of dimension 0 over $k$ \cite{pink.wedhorn.ziegler.2}.

\begin{remark}\label{g-zip.field.of.def}
If $\mu$ is defined over some finite extension $\kappa$ of $\mathbb{F}_p$, then by a similar construction we obtain a stack of $G$-zips defined over $\kappa$, which is smooth and of dimension 0 (ibid). Since our main result concerns the geometric special fiber of Hilbert modular varieties, we only study stacks of $G$-zips over $k$, and for simplicity we have chosen to omit $k$ in the notation of the stack of $G$-zips.
\end{remark}

Given a Borel pair $T\subseteq B \subseteq G$ such that $B_k\subseteq P$, we also define:
\begin{definition}
A {\bf $G$-zip flag of type $\mu$} over $S$ is a $G$-zip, $\underline{I}= (I, I_P, I_{(P^{+})^{(p)}}, \iota)$ together with a $B_k$-torsor $J \subseteq I_{P}$.
\end{definition}
Again we obtain a smooth stack, denoted $\gzipf^{\mu}$, where $\gzipf^{\mu}(S)$ is the category of $G$-zip flags of type $\mu$ over $S$.
\subsubsection{$\gzip^{\mu}$ and $\gzipf^{\mu}$ are quotient stacks}
As the header suggests, we can represent both $\gzip^{\mu}$ and $\gzipf^{\mu}$ as quotient stacks. This will be important because via the associated sheaves construction (explained below) we can then obtain vector bundles on the stacks from representations of certain groups. 

To express $G\zip^{\mu}$ as a quotient stack, let $E$ be the group defined by the following Cartesian square,
\begin{equation}
    \begin{tikzcd}
    E \arrow[rr] \arrow[d] & & (P^{+})^{(p)} \arrow[d] \\
    P \arrow[r] & P/R_u(P) \cong L \arrow[r, "\fr"] & L^{(p)} \cong (P^{+})^{(p)}/(R_u((P^{+})^{(p)}))
    \end{tikzcd},
\end{equation}
where $\text{Fr}$ denotes the Frobenius homomorphism and the maps $P\to L$ and $(P^+)^{(p)}\to (P^{+})^{(p)}/(R_{u}((P^{+})^{(p)}))$ are the natural projections. 
Less obscurely, $E$ is given by
\begin{equation}
     E = \{ (a,b) \in P\times (P^{+})^{(p)} : \fr(\overline{a}) = \overline{b} \},
\end{equation}
where $\overline{(-)}$ denotes the natural quotient maps depicted in the diagram above. The group $E$ acts on $G_k$ by $(a,b)\cdot g \coloneqq agb^{-1}$, and the group $E\times B_k$ acts on $G_k\times P$ by $((a,b),c)\cdot (g,r)\coloneqq (agb^{-1}, arc^{-1})$. We have the following theorem:
\begin{theorem}
There is a commutative diagram
\begin{equation}
    \begin{tikzcd}
    \gzipf^{\mu} \arrow[r, "\pi"] \arrow[d, "\cong"] & \gzip^{\mu} \arrow[d, "\cong"] \\
    \left[(E\times B_k)\backslash (G_k\times P)\right] \arrow[r, "\pi'"] & \left[ E\backslash G_k \right]
    \end{tikzcd},
\end{equation}
where the vertical maps are isomorphisms and the horizontal maps are the natural projections obtained from $(\underline{I},J)\mapsto \underline{I}$ and $G_k\times P\to G_k$ and $E\times B_k \to E$, respectively.
\end{theorem}
\begin{proof}
See \cite{wushi.invent}, Theorem 2.1.2.
\end{proof}
There is another representation of $G\zipf^{\mu}$ as a quotient stack which will be useful for us. Let $\Phi$ denote the roots with respect to $T$ and let $\Phi^{+}\subseteq \Phi$ be the positive roots determined by $B$ (we declare $\alpha \in \Phi$ to be positive if $U_{-\alpha}\subseteq B$, where $U_{-\alpha}$ is the root group associated to $-\alpha$). Let also $\Delta$ denote the set of simple roots. Further, let $W=W(G_k,T_k)$ be the Weyl group of $G$ with respect to $T$. It is generated by the reflections $s_{\alpha}$, $\alpha \in \Delta$. Recall that the power set of $\Delta$ is in bijection with the set of parabolics of $G$ containing $B$. Under this bijection, if $I^{+}$ denotes the type of $(P^{+})^{(p)}$, let $W_{I^{+}}$ denote the the subgroup of $W$ generated by $s_{\alpha}$ with $\alpha \in I^{+}$. Let $w_0$ denote the longest element of $W$ and let $w_{0,I^{+}}$ denote the longest element of $W_{I^{+}}$. For any $w\in W$, we choose a lift $\Dot{w}\in N_G(T)$ such that $\Dot{(w_1w_2)}=\Dot{w}_1\Dot{w}_2$ whenever $l(w_1w_2)=l(w_1)+l(w_2)$ (this is possible by \cite[Exp. XXIII, Section 6]{sga3}). 
By abuse of notation we will omit the $\cdot$ and denote $\Dot{w}$ simply by $w$; it should be clear from context whether $w$ is viewed as an element of $W$ or of $N_G(T)$. 
% also a lift of $w$ to $N_G(T)\subset G$ compatible with the length function on $W$. 
Define finally $z \coloneqq w_0w_{0,I^{+}}\in W$. Then one can show that, with $E'\coloneqq E\cap (B_k\times G_k)$, one has $E'\subseteq B_k\times {}^{z}B_k$ and $\gzipf^{\mu}\cong [E'\backslash G_k]$. This will allow us to relate $G\zipf^{\mu}$ with the Schubert stack, introduced next.
\subsection{The Schubert stack}
The group $B_k \times B_k$ acts on $G_k$ via $(a,b)\cdot g = agb^{-1}$ and we call the quotient stack the {\bf Schubert stack}, denoted
\begin{equation}
    \sbt \coloneqq [(B_k \times B_k)\backslash G_k].
\end{equation}
The map $g \mapsto gz$ induces an isomorphism $B_k \times B_k\cong B_k\times {}^{z}B_k$, whence an isomorphism $\sbt \cong [(B_k\times {}^{z}B_k)\backslash G_k]$. Composing with the natural projection $\gzipf^{\mu}\cong [E'\backslash G_k] \to [(B_k\times {}^{z}B_k)\backslash G_k]$ we obtain a smooth morphism
\begin{equation}
    \psi \colon \gzipf^{\mu} \to \sbt.
\end{equation}
Thus, the three stacks fit together in the following diagram
\begin{equation}\label{three.stacks}
    \begin{tikzcd}
    \gzip^{\mu} & \gzipf^{\mu} \arrow[l, "\pi"'] \arrow[d, "\psi"] \\
     & \sbt
    \end{tikzcd}.
\end{equation}
%
%
% STRATIFICATIONS
\subsection{Stratifications}
The Bruhat stratification allows us to write $G$ as the disjoint union of locally closed subschemes,
\begin{equation}\label{g.is.union.of.b}
    G = \coprod_{w\in W} BwB.
\end{equation}
Define $\sbt_w \coloneqq [(B_k\times B_k)\backslash B_kwB_k]$ and $\overline{\sbt}_w \coloneqq [(B_k\times B_k)\backslash \overline{B_kwB_k}]$ and then $\mathcal{X}_w \coloneqq \psi^{-1}(\sbt_w)$ and $\overline{\mathcal{X}}_w \coloneqq \psi^{-1}(\overline{\sbt}_w)$. We obtain thus a stratification of $\sbt$ and $G\zipf^{\mu}$ into locally closed substacks. There is also a related stratification of $G\zip^{\mu}$ into locally closed substacks $[E\backslash G_w]$ for certain locally closed subschemes $G_w\subseteq G$. In the Hilbert case considered in this text, the stack of $G$-zips and the stack of $G$-zip flags coincide and the stratifications agree, hence we simply refer to \cite{wushi.invent} for more details.
% THE BRUHAT DECOMP rel. to SCHUBERT
%\subsubsection{The Bruhat Decomposition}
%Closely related to the Schubert stack there is also a decomposition of $G/B$ induced from the decomposition (\ref{g.is.union.of.b}). Let $f : G\to G/B$ denotes the natural projection and let $\alpha : (B\times B)\times G \to B$ denote the map $\alpha((a,b),g) \coloneqq a$. Then we see that
%\begin{equation}
%    f((a,b)\cdot g) = \overline{agb^{-1}} = \overline{ag} = \alpha((a,b),g)f(g),
%\end{equation}
%and
%\begin{equation}
%    \alpha\Big((a,b)(a',b'), g\Big) = \alpha\Big((aa',bb'),g\Big) = aa' = \alpha((a,b),(a,b)\cdot g)\alpha((a',b'),g),
%\end{equation}
%so the pair $(f,\alpha)$ induces a morphism $\tilde{f} : \sbt\to [B\backslash (G/B)]$ and a commutative diagram
%\begin{equation}
%    \begin{tikzcd}
%    G \arrow[d] \arrow[r, "f"] & G/B \arrow[d] \\
%    \sbt \arrow[r, "\tilde{f}"] & \left[ B\backslash (G/B) \right].
%    \end{tikzcd}
%\end{equation}
%The morphism $\tilde{f}$ is in fact an isomorphism \textcolor{red}{Check this!}.
%
% <--------------------- Do I need above section???
%
%
%
%
% Associated sheaves
\section{Associated Sheaves}\label{section.ass.sheaves}
Let $X$ be a $k$-scheme and $H$ an algebraic group acting on the right on $X$. Let $\text{pr} \colon X\to [X/H]$ denote the natural projection. For every representation $r \colon H\to GL(V)$ we can form the {\bf associated sheaf} $\mathcal{L}(r)$ on $[X/H]$ defined on open subsets $U\subseteq [X/H]$ as
\begin{equation}
    \mathcal{L}(r)(U) \coloneqq \{ f \in \Hom(\text{pr}^{-1}(U), V) : f(xh) = r(h)^{-1}f(x), \forall x \in X, h \in H \}.
\end{equation}
We can view the right action as a left action via $h x \coloneqq xh^{-1}$ and identify $[H\backslash X]$ with $[X/H]$ under this action. With this notation we see that $f\in \mathcal{L}(r)(U)$ satisfies $f(hx) = r(h)f(x)$.

If $\lambda \colon H \to \mg_m$ is a character of $H$, then $\mathcal{L}(\lambda)$ is a line bundle on $[X/H]$ and we will use the following notation for its set of global sections
\begin{equation}
    H^0(\lambda) \coloneqq H^{0}(X/H, \mathcal{L}(\lambda)) = \{ f \in \mathcal{O}_X(X) : f(xh) = \lambda(h)^{-1}f(x),  \forall x \in X, h \in H \}.
\end{equation}
In particular, if $X = G$ is a reductive group and $H=B$ is a Borel subgroup of $G$ then $H^{0}(\lambda)$ carries an action of $G$ defined by $(g\cdot f)(x) = f(g^{-1}x)$. For a character $\mu$ of $T$ let $H^{0}(\lambda)_{\mu}$ denote the weight space of $\mu$, i.e. 
\begin{equation}
    H^0(\lambda)_\mu = \{ f \in H^0(\lambda) : t\cdot f = \mu(t) f, \, \forall t \in T \}.
\end{equation}
\subsection{Vector bundles on quotient stacks}
Let $\lambda \in X^{*}(L)$ be a character of $L$ and let $\mu, \nu \in X^{*}(T)$ be characters of $T$. Via the first projection $E\to P$ followed by the natural projection $P\to L$ we view $\lambda$ as an element of $X^{*}(E)$. Similarly via the first projection $E'\to B_k$ followed by the natural projection $B_k\to T_k$ we may view $\mu$ as an element of $X^{*}(E')$. The projections $B_k\to T_k$ allow us to view $\mu$ and $\nu$ as elements of $X^{*}(B)$, so $(\mu,\nu)$ is a character of $B\times B$. From the associated sheaves construction, we may therefore form the sheaves
\begin{equation}
    \begin{aligned}
    \mathcal{L}_{\gzip}(\lambda) &\text{ on } \text{$G$-$\mathtt{Zip}^{\mu}$} \\
    \mathcal{L}_{\gzipf}(\mu) &\text{ on } \text{$G$-$\mathtt{ZipFlag}^{\mu}$} \\
    \mathcal{L}_{\sbt}(\mu, \nu) &\text{ on } \sbt.%=\left[ (B_k\times B_k)\backslash G_k \right].
    \end{aligned}
\end{equation}
From the picture (\ref{three.stacks}) we can pull back via $\pi$ and $\psi$, and doing so the line bundles on the three stacks are related as follows.
\begin{lemma}\label{schubert.gzipflag.rel}
Let $\sigma \colon k \to k$ denote the inverse of the map $x\mapsto x^{p}$, and let ${}^{\sigma}(-)$ denote the Galois action on the characters of $T$. Then the following hold.
\begin{enumerate}
    \item For all $\lambda \in X^{*}(L)$, $\pi^{*}\mathcal{L}_{\gzip}(\lambda) = \mathcal{L}_{\gzipf}(\lambda)$.
    \item For all $\mu, \nu \in X^{*}(T)$, %if $\sigma : k\to k$ denotes the inverse of the map $x\mapsto x^{p}$, then we have that
    \begin{equation}\label{psi.pullback}
        \psi^{*}\mathcal{L}_{\sbt}(\mu, \nu) = \mathcal{L}_{\gzipf}(\mu + p{}^{\sigma^{-1}}(z\nu)).
    \end{equation}
\end{enumerate}
%where ${}^{\sigma}(-)$ denotes the Galois action on the characters.
\end{lemma}
\begin{proof}
See \cite[Lemma 3.1.1]{wushi.invent}.
\end{proof}
%
% ORDER OF VANISHING ON STACKS
\subsection{Order of Vanishing on Stacks in the Smooth case}
Let $S$ be a smooth scheme over $k$, $\mathcal{L}$ a line bundle on $S$ and $f$ a global section of $\mathcal{L}$. Recall that $f$ determines an ideal sheaf $\mathcal{I}$ given by the image of the morphism $\mathcal{L}^{-1}\to \mathcal{O}_S$ induced by $f$. The zero scheme of $f$ is defined to be the subscheme determined by $\mathcal{I}$ (cf. \cite{stacks-project}). %, \cite[Appendix A, C.6]{hartshorne}). 
%If $I$ denotes the ideal sheaf generated by $f$, then f
For any $n\in \mathbb{N}$, we say that $f$ vanishes to order $n$ at a closed point $s\in S$ if $\mathcal{I}_s$ is contained in $\mathfrak{m}_s^n$, where $\fkm_s$ is the maximal ideal of $\mathcal{O}_{S,s}$. The order of vanishing of $f$ at $s$ is defined to be the maximal such integer $n$, denoted by $\ord_s(f)$. Next, let $[X/H]$ be any of the stacks described above, let $\chi \colon H\to \mg_m$ be a character of $H$ and let $f \in H^0(\chi)$. 
%If $I$ denotes the ideal sheaf generated by $f$, then for $n\in \mathbb{N}$, we say that $f$ vanishes to order $n$ at a closed point $x\in X$ if $I_x$ is contained in $\mathfrak{m}_x^n$, where $\fkm$ is the maximal ideal of $\mathcal{O}_{X,x}$. The order of vanishing of $f$ at $x$ is defined to be the maximal such integer $n$, denoted by $\ord_x(f)$.
The order of vanishing of $f$ on $\overline{x}\in [X/H]$ is defined to be the order of vanishing of $f$ on $x \in X$, where $x$ is a lift $\overline{x}$. This is well-defined since $f(xh)=\chi^{-1}(h)f(x)$ for all $x\in X$ and $h \in H$. Furthermore, if $S$ is a smooth scheme and $\zeta \colon S\to [X/H]$ is a smooth morphism of stacks, then $\ord_s(\zeta^{*}f) = \ord_{\zeta(s)}(f)$. Indeed, taking the base change of $\zeta$ along the projection $X\to [X/H]$ we obtain an $H$-torsor $S'\to S$ sitting in the following Cartesian diagram
\begin{equation}
    \begin{tikzcd}
    S' \arrow[d, "\zeta^{\#}"] \arrow[r, "\pi^{\#}"] & S \arrow[d, "\zeta"] \\
    X \arrow[r, "\pi"] & \left[ X/H \right]
    \end{tikzcd}.
\end{equation}
Then we see that, for $t\in S'$ such that $\pi^{\#}(t)=s$ and with $\overline{x} = \zeta(s)$, 
\begin{equation}
    \ord_s(\zeta^{*}f) = \ord_{t}((\pi^{\#})^{*}\zeta^{*}f) = \ord_t((\zeta^{\#})^{*}\pi^{*}f) = \ord_x(\pi^{*}f) = \ord_{\overline{x}}(f).
\end{equation} 
%
%
%
%
% Shimura varieties
\section{Shimura varieties mod $p$} \label{section.shimura.var}
\subsection{Shimura varieties of Hodge type}
Let $(\mathbf{G},X)$ be a Shimura datum of Hodge type with reflex field $E$. Recall that this means that there is an embedding of Shimura data, $\varphi \colon (\mathbf{G}, X)\hookrightarrow (GSp(V,\psi), X_g)$, where $V$ is a $2g$-dimensional $\mathbb{Q}$-vector space, $\psi$ is a symplectic pairing on $V$, and $X_g$ is the Siegel double half-space. 

Let $p$ be a prime over which $\mathbf{G}$ is unramified, and let $\mathfrak{p}\subseteq \calo_E$ denote a prime over $p$. Let $K$ be a sufficiently small open compact subgroup of $\mathbf{G}(\mathbb{A}_{f})$ such that $K = K^{p}K_{p}$, where $K^{p}\subseteq \mathbf{G}(\mathbb{A}_{f}^{p})$ and $K_p$ is a hyperspecial subgroup of $G(\mathbb{Q}_p)$. Here $\mathbb{A}_f$ denotes the finite adeles and $\mathbb{A}_f^{p}$ is the subset with trivial $p$-adic component. Let $\text{Sh}_K(\mathbf{G},X)$ denote the corresponding Shimura variety over $E$, sometimes referred to as the Shimura varity of $(\mathbf{G},X)$ of level $K$. Kisin and Vasiu (see \cite{kisin} respectively \cite{vasiu}) have (independently from one another) shown that there is a canonical model, $\mathcal{S}_K(\mathbf{G},X)$, defined over $\mathcal{O}_{E,\mathfrak{p}}$. Let $S_K = \mathcal{S}_K(\mathbf{G},X)\otimes_{\mathcal{O}_{E,\mathfrak{p}}} \overline{\kappa(\mathfrak{p})}$ denote the geometric special fiber of $\mathcal{S}_K(\mathbf{G},X)$ at $\mathfrak{p}$. If $\mathcal{G}$ denotes a reductive model of $\mathbf{G}_{\mathbb{Q}_p}$ over $\mz_p$, then let $G \coloneqq \mathcal{G} \otimes_{\mz_p} \mathbb{F}_p$. 
%
%Since $\mathcal{G}$ is quasi-split, w

If $h_0 \in X$ is a conjugacy class of morphisms $\mathbb{S}\to \mathbf{G}_{\mr}$, and $\mu_0 \colon \mg_{m,\mc}\to \mathbb{S}_{\mc}$ is the cocharacter of $\mathbb{S}$ defined by $\mu_0(z) = (z,1)$, then the class of $\mu \coloneqq h_0\circ \mu_0$ is a conjugacy class of cocharacters of $\mathbf{G}$  
% We can lift the conjugacy class of $\mu$ 
and furthermore there is a representative defined over $\mathcal{O}_{E_{\mathfrak{p}}}$. Taking the special fibre of the representative we obtain a cocharacter of $G$, which we still denote by $\mu$. 

If we make choices compatible with the embedding $\varphi$ (e.g., such that the neat open compact $K_g \subseteq GSp(V,\psi)(\mathbb{A}_f)$ contains $\varphi(K)$, and that $h_g = \varphi \circ h$, etc.) then the embedding of Shimura data gives an embedding of schemes, $\varphi_K \colon \text{Sh}_K(\mathbf{G},X) \to \text{Sh}_{K_g}(GSp(V,\psi), X_g)$. This embedding lifts to the models, still denoted by $\varphi_K \colon \mathcal{S}_{K}(\mathbf{G},X)\hookrightarrow \mathcal{S}_{K_g}(GSp(V,\psi),X_g)$. If we let $\mathcal{A}_{K_g}\to \mathcal{S}_{K_g}(GSp(V,\psi),X_g)$ denote the universal abelian variety over the (model of the) Siegel modular variety, then pulling back along $\varphi_K$ and taking 
the 
geometric 
special fiber, we obtain a universal abelian variety
\begin{equation*}
    \pi_K \colon A_K \to S_K.
\end{equation*}
Zhang has shown in \cite{zhang} that there is a smooth morphism $\zeta_K \colon S_K \to \gzip^{\mu}$ whose fibers are the Ekedahl-Oort strata of $S_K$. %The geometric fibers of $\zeta_K$ are the Ekedahl-Oort strata of $S_K$. 
Combining with (\ref{three.stacks}) % and tensoring with  $k$ 
we arrive at the following diagram
\begin{equation}\label{shimura.to.gzip}
    \begin{tikzcd}
    S_K \arrow[d, "\zeta_K"] & \\
    G \zip^{\mu} & G\zipf^{\mu} \arrow[l, "\pi"] \arrow[d, "\psi"] \\
     & \sbt
    \end{tikzcd}.
\end{equation}

\begin{remark}
In fact, the morphism $\zeta_K$ is defined and smooth over a finite extension of $\mathbb{F}_p$. See also Remark \ref{g-zip.field.of.def}.
\end{remark}
\subsubsection{The Hasse invariant}
Keep the notations of the previous section, but omit the subscript $K$. The universal abelian variety over $S$ gives the Frobenius diagram
\begin{equation}
    \begin{tikzcd}
    A \arrow[r, "F"] \arrow[rr, bend left = 30, "F_{\text{abs}}"] \arrow[rd, "\pi"] & A^{(p)} \arrow[r] \arrow[d, "\pi^{(p)}"] & A \arrow[d, "\pi"] \\
     & S \arrow[r, "F_{\text{abs}}"] & S
    \end{tikzcd},
\end{equation}
where $F_{\text{abs}}$ denotes the absolute Frobenius and
where the square is Cartesian. There are (see \cite[Section 7]{katz}) two spectral sequences converging to the de Rham cohomology of $\pi \colon A\to S$, called respectively the Hodge and conjugate spectral sequences;
\begin{equation}
    \begin{aligned}
        E_{1}^{a,b} = R^{b}\pi_{*}\Omega_{A/S}^{a} &\implies H_{\dr}^{a+b}(A/S) \\
        {}_{\Conj}E_{2}^{a,b} = R^{a}\pi_{*}^{(p)} \mathcal{H}^{b}(F_{*}\Omega_{A/S}^{\bullet}) &\implies H_{\dr}^{a+b}(A/S).
    \end{aligned}
\end{equation}
There is also for each $i$ a Cartier isomorphism
\begin{equation}
    C^{-1} : \Omega_{A^{(p)}/S}^{i} \xrightarrow{\cong} \mathcal{H}^{i}(F_{*}\Omega_{A/S}^{\bullet}),
\end{equation}
from which we can rewrite the second spectral sequence as
\begin{equation}
    {}_{\text{conj}}E_{2}^{a,b} = R^{a}\pi_{*}^{(p)}\Omega_{A^{(p)}/S}^{b} \implies H_{\dr}^{a+b}(A/S).
\end{equation}
Since $S$ is smooth $F_{\text{abs}}$ is flat and hence (see \cite[III. Proposition 9.3]{hartshorne}) we have that
\begin{equation}
    R^{a}\pi_{*}^{(p)}\Omega_{A^{(p)}/S}^{b} \cong F_{\text{abs}}^{*}R^{a}\pi_{*}\Omega_{A/S}^{b}
\end{equation}
In other words, we obtain that
\begin{equation}
     {}_{\text{conj}}E_{2}^{b,a} \cong F_{\text{abs}}^{*}E_{1}^{a,b}.
\end{equation}
Since $A\to S$ is an abelian variety, the sheaves $R^b\pi_* \Omega_{A/S}^a$ are locally free and the Hodge spectral sequence degenerates at $E_1$. It follows that the conjugate spectral sequence degenerates at $E_2$, that the formation of the two spectral sequences commutes with base change $S' \to S$ and that the corresponding filtrations consists of local direct factors (see e.g. \cite[Section 6]{moonen.wedhorn}).

Now let $n$ denote the dimension of $A/S$, let $\fil^{\bullet}$ be the (decreasing) Hodge filtration of $H_{\dr}^{n}(A/S)$ corresponding to the Hodge spectral sequence, and let ${}_{\Conj}\fil_{\bullet}$ denote the (increasing) conjugate filtration of $H_{\dr}^{n}(A/S)$ corresponding to the conjugate spectral sequence. With $\omega \coloneqq \pi_{*}\Omega_{A/S}^{n}$, we have that $\fil^0/\fil^1 \cong E^{0,n} \cong R^{n}\pi_{*}\mathcal{O}_A \cong \omega^{\vee}$ and ${}_{\Conj}\fil_{0} \cong {}_{\Conj}E^{0,n} \cong \pi_{*}^{(p)}\Omega_{A^{(p)}/S}^{n}$. The {\bf Hasse invariant} is defined to be the map
\begin{equation}
    h \colon F_{\text{abs}}^{*}\omega^{\vee} = F_{\text{abs}}^{*}(\fil^{0}/\fil^{1}) \xrightarrow{\cong}  {}_{\Conj}\fil_{0} \hookrightarrow H_{\dr}^{n}(A/S) \twoheadrightarrow \fil^0/\fil^1 = \omega^{\vee}.
\end{equation}
Upon tensoring with $\omega$ the Hasse invariant can be viewed as a map $\omega^{1-p}\to \mathcal{O}_S$, i.e. as a section of $\omega^{p-1}$. 
\subsubsection{The Hodge character}
We describe the Hodge character in a certain set of coordinates. To this end, let $J$ denote the $n\times n$ matrix with anti-diagonal entries equal to $1$, and 0 elsewhere, and let the coordinates of $GSp(V,\psi)$ be given by %$\psi \leftrightarrow 
$\begin{bmatrix} 0 & J \\ -J & 0 \end{bmatrix}$. Let $T_d$ denote the diagonal torus. Under the identification $X^{*}(T_d)\cong \{ (a_1,...,a_n ; c)\in \mathbb{Z}^{n+1} : \sum_{i=1}^{n}a_i \equiv c \pmod{2} \}$, $\Big(\text{diag}(t_1z,...,t_nz,t_{n}^{-1}z ,..., t_{1}^{-1}z)\mapsto t_{1}^{a_1}\cdots t_{n}^{a_n}z^{c} \Big) \leftrightarrow (a_1,...,a_n; c)$, by \cite[Section 4.11]{wushi.invent} the {\bf Hodge character} of $GSp(V,\psi)$ in these coordinates is the character
\begin{equation}
    \eta_n = (-1,...,-1; -n).
\end{equation}
The Hodge character of $\mathbf{G}$ associated to $\varphi$ is $\eta \coloneqq \varphi^{*}\eta_n$. The Hodge character also extends to the special fiber, and it is known (see e.g. ibid, cf \cite{goldring.koskivirta.quasi-constant}) that 
\begin{equation}
    \omega = \zeta^{*}\mathcal{L}_{\gzip}(\eta)
\end{equation}
and that there exists a global section $h_{\gzip}$ of $\mathcal{L}_{\gzip}((p-1)\eta)$ such that
\begin{equation}
    h = \zeta^{*}h_{\gzip}.
\end{equation}%Hasse invariant is pulled back from a section of $\mathcal{V}((p-1)\eta)$. 

For a coordinate-free description of the Hodge character see \cite[Section 4.1.9-11]{wushi.invent}.
\subsection{Hilbert modular varieties}\label{hilbert}
Now let $F$ be a totally real field extension of $\mq$ of degree $n$. From now on let $\mathbf{G}'\coloneqq \res_{F/\mq}(GL_{2,F})$ and let $\mathbf{G} \coloneqq \{ g \in \mathbf{G}' : \det(g)\in \mq^{\times} \}$. On $\mr$-points this can also be described as
\begin{equation}
    \mathbf{G}(\mr) \cong \{ (g_i) \in \prod_{i=1}^{n} GL_{2}(\mr) : \det(g_i) = \det(g_j), \, \forall i,j \}.
\end{equation}
Let $h_0 \colon \mathbb{S}\to \mathbf{G}_{\mr}$ be the map defined on $\mr$-points by
\begin{equation}
\begin{aligned}
    h_0(\mr) \colon \mathbb{S}(\mr)=\mc^{\times} &\to \mathbf{G}(\mr) \\
    x+iy &\mapsto \Big( \begin{bmatrix} x & y \\ -y & x \end{bmatrix}, ..., \begin{bmatrix} x & y \\ -y & x \end{bmatrix} \Big)%_{i}^{n}
\end{aligned}
\end{equation}
and let $X$ denote the $\mathbf{G}(\mr)$-conjugacy class of $h_0$ inside $\Hom(\mathbb{S},\mathbf{G}_{\mr})$. We obtain thus the Hilbert Shimura datum $(\mathbf{G},X)$. The cocharacter $\mu$ is described by $\mu(z) =h_0(\mc)(z, 1)= \Big( \text{diag}(z,1), ..., \text{diag}(z,1) \Big)$, %_{i}^{n}$, 
and the reflex field is $\mq$. The associated Shimura varieties are called the {\bf Hilbert modular varieties}.

Let $(G,\mu)$ denote the corresponding cocharacter datum. We see that the parabolic $P$ is simply the Borel $B$ whose $k$-points are the (product of the) lower triangular matrices. %Let $J$ denote the $d\times d$ matrix with anti-diagonal entries equal to $1$, and 0 elsewhere, and let the coordinates of $GSp(V,\psi)$ be given by $\psi \leftrightarrow \begin{bmatrix} 0 & J \\ -J & 0 \end{bmatrix}$. 
We have an embedding $G\hookrightarrow GSp_{2n}$ given over $k$ by%Over $k$, the embedding $G_k \hookrightarrow GSp_{2n}$ is given by
\begin{equation}\label{hodge.character}
\begin{aligned}
    \varphi \colon G_k &\hookrightarrow GSp_{2n,k} \\
    \Big(\begin{bmatrix} a_1 & b_1 \\ c_1 & d_1 \end{bmatrix}, ..., \begin{bmatrix} a_n & b_n \\ c_n & d_n \end{bmatrix} \Big) &\mapsto \begin{bmatrix} \text{diag}(a_1,...,a_n) & \text{diag}(b_1,...,b_n)J \\
    \text{diag}(c_n,...,c_1)J & \text{diag}(d_n,...,d_1)
    \end{bmatrix}
\end{aligned}
\end{equation}
where the right hand side is written in block form, each block matrix being of size $n\times n$.

Let also $T\subseteq B \subseteq G$ denote the maximal torus whose $k$-points are the products of the diagonal tori.
%
%
%
%
%
%
% PROOF OF THE THEOREM
\section{Proof of Theorem \ref{main.theorem}} \label{section.proof.of.theorem}
In this section we prove the main theorem of the text. We begin by describing the order of vanishing of the Hasse invariant and then show that the order of vanishing corresponds to the relative position of the Hodge and conjugate filtrations as stated.

From here onward let $S$ denote the geometric special fiber of the Hilbert variety of level $K$ as described above, omitting the subscript. Let $\pi \colon A\to S$ denote the universal abelian variety over $S$. Further, let $(G,\mu)$ and the associated groups $B,T$ be as defined above (section \ref{hilbert}).
\subsection{Strategy outline}
The strategy is quite simple; we find a section of a line bundle on $\sbt$ which on the one hand pulls back via $\psi$ and $\zeta$ to the Hasse invariant, and on the other hand has a very concrete description as a function on $G_k/B_k$, from where we immediately read off its order of vanishing at different points.
\subsection{Order of vanishing of the Hasse invariant}
Let $G'' = \prod_{i=1}^{n}GL_{2,k}$, $G' = \prod_{i=1}^{n}SL_{2,k}$, let $B_l$ be the lower triangular matrices in $GL_{2,k}$ and let $B'' = \prod_{i=1}^{n} B_l$, $B' = G'\cap B''$. Let further $T_d$ denote the diagonal matrices in $GL_{2,k}$ and $T'' = \prod_{i=1}^{n}T_d$ and $T' = G' \cap T''$. We see that $B_k = G_k\cap G''$ and $T_k = G_k\cap T''$. 

If $\omega''$ denotes the character of $GL_{2,k}$ defined by
\begin{equation}
    \omega'' : \begin{bmatrix} s_1 & 0 \\ 0 & s_2 \end{bmatrix} \mapsto s_{1}^{-1}
\end{equation}
then let, for each $i=1,...,n$, $\eta_{i}''$ denote the character of $G''$ which acts as $\omega''$ on the $i^{\text{th}}$ component and trivially on the others and let $\eta_{i}'$ and $\eta_i$ denote the restrictions of $\eta_{i}''$ to $G'$ and $G$, respectively. With $\eta'' \coloneqq \sum_{i=1}^{n} \eta_{i}''$, we see from (\ref{hodge.character}) that the {\bf Hodge character} of $G$ is $\eta \coloneqq \eta''|_{T}$. On points it is defined as
\begin{equation}
    \eta \colon \Bigg( \begin{bmatrix} t_i z & 0 \\ 0 & t_{i}^{-1}z \end{bmatrix} \Bigg)_i \mapsto \prod_{i} t_{i}^{-1}z^{-1}.
\end{equation}
Let similarly $\eta'\coloneqq \eta|_{T'}$.

We have that $W\coloneqq W(G',T')\cong W(G,T)\cong W(G'',T'') \cong \prod_{i=1}^{n}S_2$, where $S_2 = \{ \pm 1 \}$ and $w_0 = (-1,-1,...,-1)$. 

%Via the following Proposition we relate global sections of  this gives a concrete description of global sections of $\call(-\eta, w_0\eta)$ on the Schubert stack, $\text{Sbt}$.
\begin{lemma}\label{H0.sbt}
Keep the notation as above. We have that 
\begin{equation}
    H^{0}(w_0\eta)_{\eta} = H^{0}(\text{Sbt}, \call_{\sbt}(-\eta, w_0\eta)).
\end{equation} 
\end{lemma}
\begin{proof}
An element $f \in H^{0}(\sbt, \call_{\sbt}(-\eta, w_0\eta))$ is a function $f \colon G_k \to \mathbb{A}^1$ such that for all $(a,b) \in B_k\times B_k$ we have that
\begin{equation}
f(agb) = f((a,b^{-1})\cdot g) = (-\eta,w_0\eta)(a,b^{-1})f(g) = \eta(a)^{-1}w_0\eta(b)^{-1}f(g).
\end{equation}
An element $f\in H^{0}(w_0\eta)_{\eta}$ is a function $f : G_k \to \mathbb{A}^1$ such that for all $(a,b)\in B_k\times B_k$ we have that
\begin{equation}
\begin{aligned}
    f(agb) &= (w_0\eta)(b)^{-1}f(ag) && \text{(because $f \in H^0(w_0\eta)$)} \\
    &=\Big((w_0\eta )(b)^{-1}\Big)(a^{-1}\cdot f)(g) && \text{(def. of the action of $G$ on $H^0(w_0\eta)$)} \\
    &= (w_0\eta)(b)^{-1}\eta(a)^{-1}f(g) && \text{(because $f \in H^{0}(w_0\eta)_{\eta}$)}.
\end{aligned}
\end{equation}
To elaborate on the last equation; note that $H^0(w_0\eta)_\eta$ is the lowest weight space of $H^0(w_0\eta)$, hence consists of vectors invariant under the unipotent radical, $U$, of $B$. Since $\eta(a)=\eta(t)$, for $a=ut\in B_k=U_k T_k$, the equality follows. 
%The two subrepresentations of $\calo(G)$ are hence the same.
\end{proof}
\begin{remark}
More generally we see that
\begin{equation}
    H^0(\text{Sbt}, \call_{\sbt}(\lambda, -w_0\lambda)) = H^0(-w_0\lambda)_{-\lambda},
\end{equation}
for every dominant character $\lambda$.
\end{remark}
\begin{lemma}\label{res.to.sl}
The restriction map $\calo(G)\to \calo(G')$ induces an isomorphism
\begin{equation}
    \res_{G'}^{G_k} H^{0}(w_0\eta) \cong H^{0}(w_0\eta').
\end{equation}
\end{lemma}
\begin{proof}
By abuse of notation, let $w_0\eta_i$ also denote the character of $GL_{2,k}$ obtained from embedding $T_d$ into the $i^{\text{th}}$ component in $G''$ (and similarly for $w_0\eta_i'$). Then by \cite[Lemma I.3(5)]{jantzen} we see that
\begin{equation}\label{ind.tensor}
\begin{aligned}
    \ind_{B''}^{G''}(w_0\eta'')&\cong \bigotimes_{i=1}^{n} \ind_{B_l}^{GL_{2,k}}(w_0\eta_i'') \\
    \ind_{B'}^{G'}(w_0\eta') & \cong \bigotimes_{i=1}^{n} \ind_{SL_{2,k}\cap B_l}^{SL_{2,k}}(w_0\eta_i').
\end{aligned}
\end{equation}
The character $w_0\omega''$ is the character $\overline{\omega}$ in \cite[Section II.2.16]{jantzen} and the argument given there applies also to $SL_{2,k}$, which shows that $\res_{SL_{2,k}}^{GL_{2,k}}H^{0}(w_0\eta_i'') \cong H^{0}(w_0\eta_i')$, whence $H^{0}(w_0\eta') \cong \res_{G'}^{G''}H^{0}(w_0\eta'')$, by (\ref{ind.tensor}) above. We thus have the following diagram
\begin{equation}
    \begin{tikzcd}
    \res_{G'}^{G''}H^{0}(w_0\eta'') \arrow[r] \arrow[rr, bend left = 20, "\cong"] & \res_{G'}^{G_k}H^0(w_0\eta) \arrow[r] & H^0(w_0\eta'),
    \end{tikzcd}
\end{equation}
where both horizontal arrows are the (restrictions of the) restrictions $\calo(G'')\to \calo(G_k)$ and $\calo(G_k)\to \calo(G')$, respectively. Hence, the restriction map from $H^{0}(w_0\eta)$ to $H^{0}(w_0\eta')$ is surjective, and we have only to prove injectivity. 

To this end, note that under the identification $\mg_{m,k} = \{ (g_i) \in G_k : \forall i=1,...,n, \, g_i = \begin{bmatrix} a & 0 \\ 0 & 1 \end{bmatrix}, \, a \in k^{\times} \}$ the map 
\begin{equation}
    G'\times \mg_{m,k} \to G_k, \,\,\, (g', a)\mapsto g'a
\end{equation}
is an isomorphism. Now let $f_1, f_2 \in H^0(w_0\eta)$ be such that $f_1|_{G'} = f_2|_{G'}$. For any $g\in G_k$, write $g = g'a$ for some $g'\in G'$ and $a\in \mg_{m,k}$. Then, noting that $\mg_{m,k}\subseteq B_k$, we have that
\begin{equation}
    \begin{aligned}
    f_1(g) &= f_1(g'a) \\
    &=w_0\eta(a)^{-1}f_1(g') && \text{($f_1 \in H^0(w_0\eta)$)} \\
    &=w_0\eta(a)^{-1}f_1|_{G'}(g') \\
    &=w_0\eta(a)^{-1}f_2|_{G'}(g') && \text{(since $f_1|_{G'}=f_2|_{G'}$)} \\
    &=w_0\eta(a)^{-1}f_2(g') \\
    &=f_2(g'a) && \text{($f_2 \in H^0(w_0\eta)$)} \\
    &= f_2(g).
    \end{aligned}
\end{equation}
Hence, we see that $\res_{G'}^{G_k}H^0(w_0\eta) \cong H^0(w_0\eta')$. 
\end{proof}
\begin{theorem}\label{hasse.vanish}
Keep the notation as above. The Hasse invariant vanishes to exactly order $m$ on points in the codimension $m$ strata in the Ekedahl-Oort stratification.
\end{theorem}
\begin{proof}
Since $P=B$ we see that $\gzip^{\mu}=\gzipf^{\mu}$ so diagram (\ref{shimura.to.gzip}) simplifies to the following sequence of smooth morphisms
\begin{equation}
S \xrightarrow{\zeta} \gzip^{\mu} \xrightarrow{\psi} \sbt,
\end{equation}
and we identify the sheaves $\call_{\gzip}(\lambda)$ and $\call_{\gzipf}(\lambda)$, for any character $\lambda \in X^{*}(T)$. 
We know from Lemma \ref{schubert.gzipflag.rel}.2. that
\begin{equation}
    \psi^{*}\mathcal{L}_{\sbt}(-\eta, w_0\eta) = \mathcal{L}_{\gzipf}(-\eta + p{}^{\sigma^{-1}}(zw_0\eta)).
\end{equation} 
%Now, by definition we have that
%\begin{equation}
%    T_{k}' = (\res_{\mathbb{F}_{p^d}/\mathbb{F}_p}(T))_k = \prod_{\sigma} T_{\sigma}
%\end{equation}
%and the Galois action on $X^{*}(T)$ comes from permutation of the factors.
% OR: $T$ is quasi-split so the Galois action is via the permutation action, hence ${}^{\sigma}\eta = \eta$. 
Since $\eta$ is defined over $\mathbb{F}_p$ we see that ${}^{\sigma^{-1}}\eta = \eta$ and since $z=w_0$ and $w_{0}^2=1$ we thus obtain that
\begin{equation}
    \psi^{*}(\call_{\sbt}(-\eta, w_0\eta)) = \call_{\gzipf}(-\eta + p\eta)= \call_{\gzip}(-\eta + p\eta) = \call_{\gzip}((p-1)\eta).
\end{equation}
Hence,
\begin{equation}
    \omega^{p-1} = \zeta^{*}\call_{\gzip}((p-1)\eta) = \zeta^{*}\psi^{*}(\call_{\sbt}(-\eta, w_0\eta)).
\end{equation}
Since both $H^0(\sbt,\call_{\sbt}(-\eta,w_0\eta))$ and $H^0(\gzip^{\mu},\call_{\gzip}((p-1)\eta))$ are one-dimensional, we see that there exists a section $h_{\sbt}$ of $\call_{\sbt}(-\eta,w_0\eta)$ such that
\begin{equation}
    h = \zeta^{*}h_{\gzip}= \zeta^{*}\psi^{*}h_{\sbt}.
\end{equation}
Thus the conclusion follows if we show the corresponding statement for $h_{\sbt}$.

To this end, note that
\begin{equation}
    \prod_{i=1}^{n}\mathbb{P}_k^1 \cong \prod_{i=1}^{n}\frac{SL_{2,k}}{SL_{2,k}\cap B_l} \cong G'/B' \hookrightarrow G_k/B_k \hookrightarrow G''/B'' \cong \prod_{i=1}^{n}GL_{2,k}/B_l \cong \prod_{i=1}^{n} \mathbb{P}_k^1,
\end{equation}
so $G_k/B_k \cong \prod_{i=1}^{n} \mathbb{P}_k^1$, 
and if we combine (\ref{ind.tensor}) and Lemma \ref{res.to.sl} we see that $H^0(w_0\eta)$ is the global sections of the line bundle $\calo(1,1,...,1)$ on $\prod_{i=1}^{n} \mathbb{P}_k^1$. Let $[x_{i0} : x_{i1}]$ denote the projective coordinates on the $i^{\text{th}}$ component. An element of $H^0(w_0\eta)$ is then a bihomogenous polynomial, $f=f(x_{10},x_{11} ; x_{20},x_{21} ; ... ; x_{n0}, x_{n1})$ of bidegree $(1,1,...,1)$, and the action of $b = (b_i)\in B_k$ is given by
\begin{equation}
    b\cdot f (x_{10},x_{11} ; x_{20},x_{21} ; ... ; x_{n0}, x_{n1}) = f(b_{1}^{-1}\begin{bmatrix} x_{10}\\x_{11}\end{bmatrix} ; ..., b_{n}^{-1} \begin{bmatrix} x_{n0}\\x_{n1} \end{bmatrix}).
\end{equation}
Hence, using Lemma \ref{H0.sbt} we see that $h_{\sbt} \in H^{0}(\text{Sbt}, \call_{\sbt}(-\eta, w_0\eta))=H^0(w_0\eta)_{\eta}$, up to a scalar, is given by
\begin{equation}\label{expl.descr.Hasse}
    h_{\sbt} = \prod_{i=1}^{n}x_{i0}.
\end{equation}
Furthermore, under the identification $G_k/B_k\cong \prod_{i=1}^{d} \mathbb{P}_k^1$, the Schubert cell corresponding to the element $w=(x_1, ..., x_d) \in W$ is the subscheme $C(w) = \prod_{i=1}^{d}C_i$ where 
\begin{equation}
    C_i = \begin{cases} \mathbb{A}^1 & x_i = -1 \\
    \{ [0:1] \} & x_i = 1
    \end{cases}.
\end{equation}
Thus we see that $h_{\sbt}$ vanishes exactly to order $\codim \overline{C(w)} = |\{ i : x_i = 1 \} = d - l(w)$ on $\sbt_w$, where $l(w)$ denotes the length of $w$. Since $h= \zeta^{*}\psi^{*}h_{\sbt}$ we obtain the desired result.
%%%%%%%%%the Hasse invariant is pulled back from a section of $\call(-\eta,w_0\eta)$. Via Proposition \ref{H0.sbt} we thus see that the order of vanishing of the Hasse invariant can be read off from equation (\ref{expl.descr.Hasse}). % and Proposition \ref{H0.sbt} we see that the Hasse invariant vanish to exactly order $m$ on points in the codimension $m$ strata in the Ekedahl-Oort stratification. 
%%%%%%%%%%By Remark \ref{vanishing.stratification.rel} we obtain the desired result.
\end{proof}
\begin{remark}
This finishes the proof of the equivalence of 1. and 2. in Theorem \ref{main.theorem}.
\end{remark}
\begin{remark}\label{partial.hasse}
 Note that the decomposition of $h_{\sbt}$ in equation (\ref{expl.descr.Hasse}) produces sections $x_{i0}$ which are pulled back to the partial Hasse invariants (as defined e.g. in \cite{goren} or \cite{wushi.invent}). From this description, it is %conceptually relatively 
clear that their zero loci indeed are reduced, and that the order of vanishing of points where they vanish is 1.
\end{remark}
\begin{remark}\label{w.r.t.goren}
Let us further discuss the close relation of this result with the work by Goren \cite{goren}. In it, Goren constructs partial Hasse invariants which factorise the Hasse invariant $h$ studied in this text, and the zero locus of each such partial Hasse invariant is the closure of a codimension 1 Ekedahl--Oort stratum. Since the Hilbert modular variety is smooth, the order of vanishing is additive on functions and one could study the order of vanishing of each partial Hasse invariant to obtain our result. This is not done in \cite{goren}. Another approach, using the work in ibid, could be to use the knowledge of the order of vanishing of $h$ in codimension 1 to obtain results in higher codimension. %, but we see from the decomposition in  he does not study the order of vanishing at points, we can not a priori use this decomposition to obtain the order of vanishing of $h$. 
However, the order of vanishing of sections is in general not additive on intersections; %This is explained further in Appendix \ref{order.of.van.appendix} below, but 
a simple example can be found by looking at $f(x,y,z)=(x+y)xy \in \mathbb{C}[x,y,z]$ and consider its order of vanishing on $V(x)$, $V(y)$ and $V(x)\cap V(y)=V(x,y)$ respectively. Thus, it is not clear a priori from the results of Goren, what happens with the order of vanishing of the Hasse invariant on intersections of Ekedahl--Oort strata, and this is not discussed in ibid. We rigorously prove that, indeed, the order of vanishing is what one might ``hope'' from the work of Goren. 
\end{remark}
\subsection{The Hasse invariant and the filtrations}
In this section we prove the relation between the order of vanishing of the Hasse invariant at a point $s \in S$ and the relative position of the Hodge and conjugate filtrations, as stated in Theorem \ref{main.theorem}. %Thus, fix a point $s\in S$, let $A_s$ denote the fiber of $s$ and let 
For simplicity let $H_{\dr}^{j}\coloneqq H_{\dr}^{j}(A/S)$ for all $j$, and let $A_s$ denote the fiber of $s$. For any locally free sheaf $\mathcal{E}$ on $S$, let $\mathcal{E}(s)$ denote the fiber at $s$. We prove the following.
\begin{theorem}
Keep the notation as above. The Hasse invariant vanishes to order $m$ at $s\in S$ if and only if $m$ is the largest integer such that ${}_{\Conj}\fil_0H_{\dr}^{n}(A_s/k) \subseteq \fil^{m}H_{\dr}^{n}(A_s/k)$.
\end{theorem}
\begin{proof}
Given the Hodge filtration $\fil^{\bullet}$ on $H_{\dr}^1$, let $\Omega \coloneqq \fil^1$ and $M \coloneqq H_{\dr}^1/\Omega$. The filtration is then written as
\begin{equation}\label{hodge.fil.1}
    0 \to \Omega \to H_{\dr}^{1} \to M \to 0.
\end{equation}
If we let ${}_{\Conj}\fil_{\bullet}$ denote the conjugate filtration on $H_{\dr}^{1}$, and $(-)^{(p)}$ denote pullback via absolute Frobenius, then it is known that $M^{(p)} \cong {}_{\Conj}\fil_{0}$. We thus have a map $h^{1}$ as depicted in the following diagram:
\begin{equation}
    \begin{tikzcd}
    M^{(p)}\arrow[r, "\cong"] \arrow[rrr, bend left = 30, "h^{1}"] & {}_{\Conj}\fil_{0} \arrow[r, hook] & H_{\dr}^{1} \arrow[r, twoheadrightarrow] & M
    \end{tikzcd}
\end{equation}
Note that $h^{1} = 0$ if and only if ${}_{\Conj}\fil_0 \subseteq \Omega$, in which case we can view $h^{1}$ as a map from $M^{(p)}$ to $\Omega$.

The $n$ embeddings of $F$ into $\mr$ give a decomposition of $H_{\dr}^{1}$ into two-dimensional pieces $H_{\dr}^{1} = \oplus_{i=1}^{n} (H_{\dr}^{1})_{i}$ which give decompositions
\begin{equation}
    \Omega = \oplus_{i=1}^{n} \Omega_i, \,\,\, M = \oplus_{i=1}^{n} M_i, \,\,\, {}_{\Conj}\fil_{0} = \oplus_{i=1}^{n} ({}_{\Conj}\fil_{0})_i,
\end{equation}
where each factor has dimension 1. 
When $p$ is split, Frobenius induces, for each $i$, a map $h_i$ as depicted in the following diagram
\begin{equation}
    \begin{tikzcd}
    M_{i}^{(p)}\arrow[r, "\cong"] \arrow[rrr, bend left = 30, "h_i"] & ({}_{\Conj}\fil_{0})_i \arrow[r, hook] & (H_{\dr}^{1})_i \arrow[r, twoheadrightarrow] & M_i
    \end{tikzcd}
\end{equation}
and $h^{1} = \oplus_{i=1}^{n} h_i$ (when $p$ is inert, up to reordering of the indices, Frobenius induces maps $h_{i} \colon M_{i}^{(p)}\to M_{i+1}$ for $i=1,...,n-1$ and $h_n \colon M_{n}^{(p)}\to M_1$. We still have that $h^{1} = \oplus_{i=1}^{n}h_i$ and the argument below will be the same). The $h_i$ are the partial Hasse invariants, and we note that $h_i =0$ if and only if $({}_{\Conj}\fil_{0})_i \subseteq \Omega_i$.

The Hodge filtration on $H_{\dr}^{n}= \bigwedge^{n}H_{\dr}^{1}$ is induced from (\ref{hodge.fil.1}); the $k^{\text{th}}$ piece is given by
\begin{equation}
    \fil^{k}H_{\dr}^{n} = \im\Big(\Omega^{\otimes k}\otimes (H_{\dr}^{1})^{\otimes (n-k)} \to (H_{\dr}^{1})^{\otimes n} \to \bigwedge^{n} H_{\dr}^{1} \Big).
\end{equation}
Since the Hodge and conjugate filtrations are locally direct factors the $k^{\text{th}}$ graded piece is given by
\begin{equation}\label{graded.piece}
\begin{aligned}
    \gr^{k}\fil^{\bullet}H_{\dr}^{n} &= \bigwedge^{k} \Omega \otimes \bigwedge^{n-k} M \\
    &= \bigoplus_{i_1<... < i_k}\bigoplus_{j_1<... < j_{n-k}} \Omega_{i_1}\otimes\cdots\otimes \Omega_{i_k} \otimes M_{j_1}\otimes \cdots \otimes M_{j_{n-k}}.
\end{aligned}
\end{equation}
We similarly note that ${}_{\Conj}\fil_{0}H_{\dr}^{n} = \bigwedge^{n} {}_{\Conj}\fil_{0}H_{\dr}^{1}$. 
The Hasse invariant is the map $h \coloneqq \det h^{1}$, depicted in the following diagram
\begin{equation}\label{hasse.diag}
    \begin{tikzcd}
    \otimes_{i=1}^{n} M_{i}^{(p)} \arrow[r, "\cong"] \arrow[rrrdd, bend right = 20, "h=\otimes_{i=1}^{n} h_i"] &[-1em] \Big(\gr^{0}\fil^{\bullet}H_{\dr}^{n}\Big)^{(p)} \arrow[r, "\cong"] &[-1em] 
    {}_{\Conj}\fil_{0}H_{\dr}^{n} \arrow[r, hook] &[-1em]
    H_{\dr}^{n} \arrow[d, twoheadrightarrow] \\
     & & & \gr^{0}\fil^{\bullet}H_{\dr}^{n} \arrow[d, "\cong"] \\
     & & & \otimes_{i=1}^{n} M_i
    \end{tikzcd}.
\end{equation}
Since $S$ is smooth we see from (\ref{hasse.diag}) that
\begin{equation}
    \ord_s(h)=\sum_{i=1}^{n}\ord_s(h_i),
\end{equation}
and since $\ord_s(h_i)\in \{0,1\}$ for all $i$ (see Remark \ref{partial.hasse}) we see that the order of vanishing of $h$ at $s$ is equal to the number of $i$ such that $h_i(s) = 0$. % (note that if $h_i$ vanishes at $s$, then by the proof of Theorem \ref{hasse.vanish} (cf equation (\ref{expl.descr.Hasse})) it vanishes to order 1). 
Thus, $h$ vanishes to order $m$ at $s$ if and only if there are $i_1< ... < i_m$ such that $h_{i_l}(s)=0$ for all $l=1,...,m$. Since the construction of the Hodge and conjugate spectral sequences commute with base change along $s\to S$, this means that each $h_{i_l}(s)$ can be viewed as a non-zero map from $M_{i_{l}}^{(p)}(s)$ to $\Omega_{i_{l}}(s)$, and from (\ref{hasse.diag}) we see that this happens if and only if 
\begin{equation}
    {}_{\Conj}\fil_{0}H_{\dr}^{n}(A_s/k) = (\otimes_{l=1}^{m}\Omega_{i_l}(s)) \otimes (\otimes_{l=m+1}^{n}M_{i_l}(s)),
\end{equation}
where we let $i_{m+1}, ..., i_n$ be numbers in $\{ 1,...,n \}\setminus \{ i_1, ..., i_m \}$ with $i_{l+1}<...<i_n$. Looking at (\ref{graded.piece}) we see that this happens if and only if $m$ is the largest integer such that ${}_{\Conj}\fil_0H_{\dr}^{n}(A_s/k) \subseteq \fil^{m}H_{\dr}^{n}(A_s/k)$. 
\end{proof}
\begin{remark}
This finishes the proof of the equivalence between 2. and 3. in Theorem \ref{main.theorem} and thus concludes the whole theorem. 
\end{remark}

\section{Acknowledgements}
I would like to thank my advisor, Wushi Goldring, for coming up with the idea of this project and for many helpful discussions. 
I would also like to thank an anonymous referee for several useful comments that greatly improved the quality of the paper.
\newpage
\printbibliography %Prints bibliography
\end{document}